\documentclass[11pt]{amsart}
\usepackage[margin=1in]{geometry}

\usepackage{amssymb}
\usepackage{amsthm}
\usepackage{amsmath}
\usepackage{mathrsfs}
\usepackage{amsbsy}
\usepackage[all]{xy}
\usepackage{bm}
\usepackage{hyperref}
\usepackage{tikz}
\usepackage{array}
\usepackage{float}
\usepackage{enumerate}
\usepackage{xcolor}
\usepackage{hhline}
\allowdisplaybreaks
\usepackage[noadjust]{cite}

\usepackage{caption}
\usepackage{tabu}
\usepackage{diagbox}
\usepackage{comment}

\newenvironment{enumerate*}
  {\begin{enumerate}[(I)]
    \setlength{\itemsep}{10pt}
    \setlength{\parskip}{0pt}}
  {\end{enumerate}}

\newtheorem{theorem}{Theorem}[section]
\newtheorem{proposition}[theorem]{Proposition}

\newtheorem{conjecture}[theorem]{Conjecture}
\newtheorem{question}[theorem]{Question}

\newtheorem{lemma}[theorem]{Lemma}

\theoremstyle{definition}

\newtheorem{remark}[theorem]{Remark}

\DeclareMathOperator{\Cross}{Cross}
\DeclareMathOperator{\Rug}{Rug}
\DeclareMathOperator{\HComb}{HorComb}
\DeclareMathOperator{\VComb}{VerComb}
\DeclareMathOperator{\area}{area}

\DeclareMathOperator{\length}{length}

\newcommand{\longstitch}{(a,b,c,d)} 

\newcommand{\dfn}[1]{\textcolor{blue}{\emph{#1}}}

\begin{document}

\title[]{Extensions of Hitomezashi Patterns}
\subjclass[2010]{}

\author[Colin Defant]{Colin Defant$^\dagger$}
\address[]{Department of Mathematics, Massachusetts Institute of Technology, Cambridge, MA 02139, USA}
\email{colindefant@gmail.com}
\thanks{$^\dagger$Supported by the National Science Foundation under Award No. 2201907 and by a Benjamin Peirce Fellowship at Harvard University.}

\author[Noah Kravitz]{Noah Kravitz$^{\#}$}
\address[]{Department of Mathematics, Princeton University, Princeton, NJ 08540, USA}
\email{nkravitz@princeton.edu}
\thanks{$^{\#}$Research supported by an NSF Graduate Research Fellowship (grant DGE--2039656).}

\author[Bridget Eileen Tenner]{Bridget Eileen Tenner$^*$}
\address{Department of Mathematical Sciences, DePaul University, Chicago, IL, USA}
\email{bridget@math.depaul.edu}
\thanks{$^*$Research partially supported by NSF Grant DMS-2054436.}

\maketitle

\begin{abstract}
Hitomezashi, a form of traditional Japanese embroidery, gives rise to intricate arrangements of axis-parallel unit-length stitches in the plane. Pete studied these patterns in the context of percolation theory, and the first two authors recently investigated additional structural properties of them.    In this paper, we establish several optimization-style results on hitomezashi patterns and provide a complete classification of ``long-stitch'' hitomezashi patterns in which stitches have length greater than $1$.  We also study variants in which stitches can have directions not parallel to the coordinate axes.
\end{abstract}

\section{Introduction}\label{sec:intro}

\emph{Hitomezashi} is a style of Japanese embroidery in which stitches are arranged on a cloth according to rigid rules. Mathematically, we can think of the parts of the thread lying on top of the cloth (i.e., the parts that are visible when we view the cloth from above) as unit-length line segments called \dfn{stitches} in the infinite unit-square grid. More rigorously, a \dfn{hitomezashi pattern} is a collection of stitches with the property that every lattice point in the grid is the endpoint of exactly one horizontal stitch and exactly one vertical stitch. These patterns have gained a great deal of interest for their artistic beauty, but they are also beautiful mathematical structures. The deeper properties of hitomezashi patterns were first investigated by Pete \cite{Pete} in the context of percolation theory. Apparently unaware of the connection with Japanese art, Pete used the name \emph{corner percolation} to refer to hitomezashi patterns. We learned about hitomezashi patterns from the video \cite{Numberphile} on Brady Haran's YouTube channel \emph{Numberphile}, in which Ayliean MacDonald describes how to construct hitomezashi patterns and suggests extending the definition to the triangular grid (Pete also suggested a different extension to the triangular grid in \cite{Pete}). We refer the reader to \cite{HayesSeaton2, HayesSeaton} for further discussions of the cultural aspects of hitomezashi and to \cite{Holden} for other forms of mathematically interesting embroidery.

The bounded connected components of the union of the stitches in a hitomezashi pattern are called \dfn{hitomezashi loops}. Pete \cite{Pete} proved several structural results about the shapes of hitomezashi loops, including a correspondence between hitomezashi loops (modulo translation) and pairs of Dyck paths of the same height. One immediate consequence of this correspondence is that all hitomezashi loops have odd height and odd width. He also established several fascinating probabilistic results about random hitomezashi patterns. In \cite{defant kravitz}, the first two authors showed that every hitomezashi loop has length congruent to $4$ modulo $8$ and bounds a region with area congruent to $1$ modulo $4$. They also proved that in a random hitomezashi pattern (chosen according to a natural random model), the average area bounded by a hitomezashi loop is $12/(\pi^2-9)$.

In this paper, we study Hitomezashi patterns in further directions, which can be seen as motivated by their artistic origins. Textile artists work with materials, and questions of maximizing and minimizing the length of a loop and the area that it encloses (that is, optimizing measures on the loop) arise easily. Similarly, since stitch variation is common in textile arts (for example, cables and wrapped stitches in knitting, different warp and treadle arrangements in weaving), it is natural to study this in the context of Hitomezashi loops.

In Section~\ref{sec:optimization}, we resolve several optimization-style problems about hitomezashi loops. In each of these, we fix the width and height of the loop and then prove sharp bounds on its length and area.  We also characterize when equality is attained. The results on minimizing length (Theorem~\ref{thm:min-length}) and area (Theorem~\ref{thm:min-area}) are elegant and perhaps what one would expect. The result on maximizing length (Theorem~\ref{thm:max-length otherwise}) is substantially more complicated, and the result on maximizing area (Theorem~\ref{thm:max-area}) is easy both to state and to prove.

In Section~\ref{sec:long-square}, we study the following generalization of hitomezashi patterns in which we allow stitches of length greater than $1$: For positive integers $a, b, c, d$, we can consider sewing the horizontal stitches in an ``$a$-over-$b$-under'' pattern and sewing the vertical stitches in a ``$c$-over-$d$-under'' pattern.  For $u,v\in\mathbb R^2$, let $[u,v]$ denote the line segment that has $u$ and $v$ as its endpoints. For $i\in\mathbb Z$, let us label the horizontal grid line $y=i$ with a label $\varepsilon_i\in\mathbb Z/(a+b)\mathbb Z$; the horizontal stitches on this grid line are the line segments of the form $[(k,i),(k+a,i)]$ such that $k\equiv\varepsilon_i\pmod{a+b}$. Likewise, for $j\in\mathbb Z$, we label the vertical grid line $x=j$ with a label $\eta_j\in\mathbb Z/(c+d)\mathbb Z$; the vertical stitches on this grid line are the line segments of the form $[(j,k),(j,k+c)]$ such that $k\equiv\eta_j\pmod{c+d}$. We say the resulting arrangement of stitches is an \dfn{$(a,b,c,d)$-hitomezashi pattern} if every endpoint of a vertical stitch is the endpoint of a horizontal stitch and vice versa. See Figure~\ref{FigHitoB1} for an assignment of grid labels that satisfies this compatibility condition and an assignment of grid labels that violates it. A $(1,1,1,1)$-hitomezashi pattern is an ordinary hitomezashi pattern, and any assignment of grid labels will automatically satisfy the compatibility condition in this case.  

\begin{figure}[ht]
  \begin{center}\includegraphics[height=4.5cm]{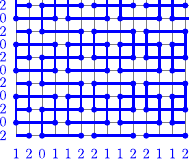}\qquad\qquad\qquad\qquad\includegraphics[height=4.5cm]{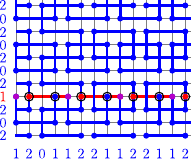}
  \end{center}
  \caption{On the left is a finite portion of a $(3,1,2,2)$-hitomezashi pattern. On the right is an assignment of grid labels that cannot be part of a $(3,1,2,2)$-hitomezashi pattern because it violates the compatibility condition; black circles indicate the violations. }\label{FigHitoB1}
\end{figure}

We will determine which quadruples $(a,b,c,d)$ can lead to $(a,b,c,d)$-hitomezashi patterns, and we will completely characterize such patterns  (Theorems~\ref{thm:classification-generic} and~\ref{thm:classification-a=b}).  One surprising consequence is that $4$-stitch rectangles are the only loops that can appear.  Another surprising consequence is (essentially) that the number of $(a,b,c,d)$-hitomezashi patterns is finite if and only if $a \neq b$ and $c \neq d$; in this case, we give an exact enumeration of the $(a,b,c,d)$-hitomezashi patterns. It is interesting to note that all such $(a,b,c,d)$-hitomezashi patterns are periodic in both the $x$- and $y$-directions.  One can view the $a=b$, $c \neq d$ case (equivalently, the $a \neq b$, $c=d$ case) as intermediate between the $a \neq b$, $c\neq d$ case, which is very rigid, and the $a=b=c=d$ case (normal hitomezashi), which has a great variety of patterns.

In Section~\ref{sec:long-triangular}, we consider the long-stitches variation on the triangular grid.  We focus on the case where each direction follows the $a$-over-$b$-under schema, and we find that (up to translation) there is only a single possible hitomezashi pattern (Theorem~\ref{thm:classification-triangular}).  This result again stands in contrast to ordinary triangular hitomezashi patterns, which exhibit great diversity and seem difficult to understand.

Finally, in Section~\ref{sec:further}, we present conjectures and open problems for future inquiry.

\section{Optimization on the Square Grid}\label{sec:optimization}

The hitomezashi loops of a given width and height can look quite different from one another, as demonstrated in Figure~\ref{FigHitoB2}. The purpose of this section is to determine how big or small the length and area of a hitomezashi loop can be if we fix the width and height beforehand.  We begin by setting up some notation and terminology.  The \dfn{longitude} (respectively, \dfn{latitude}) of a stitch is the $x$-coordinate (respectively, $y$-coordinate) of its midpoint. The \dfn{width} (respectively, \dfn{height}) of a hitomezashi loop $L$ is the maximum difference between the longitudes (respectively, latitudes) of the stitches in $L$. We say a vertical stitch in $L$ is \dfn{west-extremal} (respectively, \dfn{east-extremal}) if its longitude is minimal (respectively, maximal) among all stitches in $L$. We define \dfn{south-extremal} and \dfn{north-extremal} horizontal stitches similarly. Let $\length(L)$ and $\area(L)$ denote the length (i.e., number of stitches) of $L$ and the area of the region enclosed by $L$, respectively.  We can now state the following important results due to Pete.

\begin{figure}[ht]
 \centering
 \includegraphics[width=\linewidth]{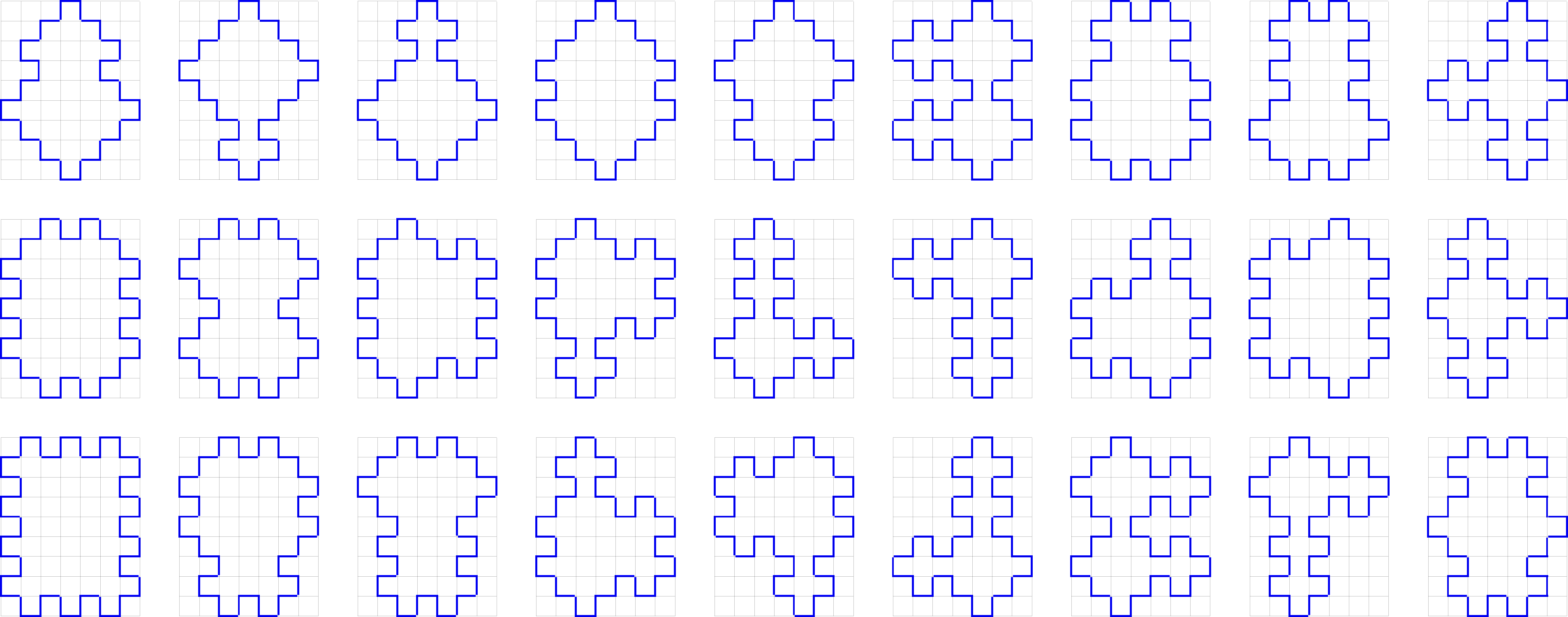}
  \caption{The $27$ hitomezashi loops having width $7$ and height $9$.}\label{FigHitoB2}
\end{figure}

\begin{theorem}[Pete, \cite{Pete}]\label{thm:Pete_odd}
Every hitomezashi loop has odd width and odd height. 
\end{theorem}

\begin{theorem}[Pete, \cite{Pete}]\label{thm:Pete_extremal}
A hitomezashi loop $L$ has a west-extremal vertical stitch at latitude $y$ if and only if it has an east-extremal vertical stitch at latitude $y$. If $L$ has west-extremal and east-extremal stitches at latitude $y$, then these are the only two (vertical) stitches of $L$ at latitude $y$. The analogous statements hold for horizontal stitches.  
\end{theorem}

Theorem~\ref{thm:Pete_extremal} motivates us to define an \dfn{extremal latitude} (respectively, \dfn{extremal longitude}) of a hitomezashi loop $L$ to be a latitude (respectively, longitude) at which $L$ has west-extremal and east-extremal vertical stitches (respectively, south-extremal and north-extremal horizontal stitches).

To describe the equality cases for several of our results in this section, we will need names for some distinguished types of hitomezashi loops. First, for odd positive integers $w$ and $h$, consider the hitomezashi pattern in the rectangular region $[0,w]\times [0,h]$ having the following grid labels: we put $\varepsilon_0=\varepsilon_h=\eta_0=\eta_w=1$ and set all other grid labels equal to $0$. This hitomezashi pattern has one loop of width $w$ and height $h$; we define the $w\times h$ \dfn{rug}, denoted $\Rug_{w\times h}$, to be this hitomezashi loop, considered modulo translation. The image on the left of Figure~\ref{fig:rug} shows $\Rug_{11\times 13}$.

\begin{figure}[ht]
  \begin{center}
  \includegraphics[height=5cm]{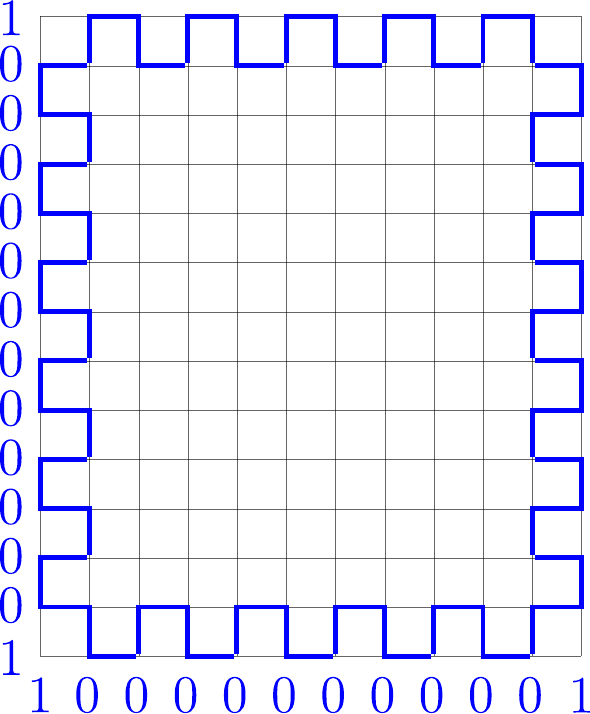}\qquad\qquad\includegraphics[height=5cm]{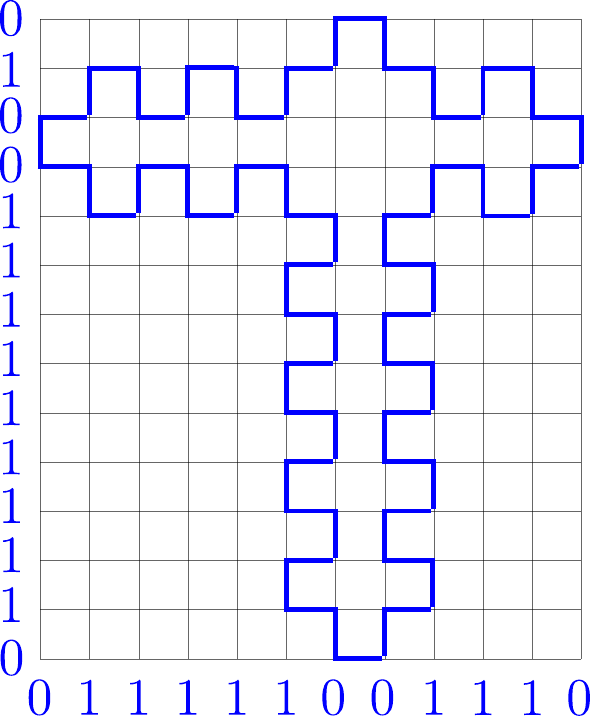}\qquad\qquad\includegraphics[height=5cm]{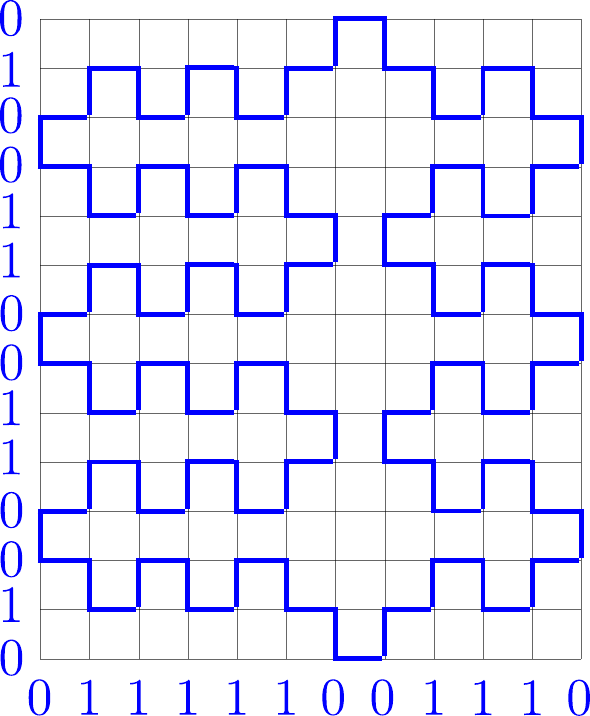}
  \end{center}
  \caption{On the left is the rug $\Rug_{11\times 13}$. In the middle is the cross $\Cross^{7,11}_{11\times 13}$. On the right is the horizontal comb $\HComb_{11\times 13}^7$. }\label{fig:rug}
\end{figure}

Second, suppose $w,h\geq 5$ are odd, and choose odd integers $3\leq\alpha\leq w-2$ and $3\leq \beta\leq h-2$. Consider the hitomezashi pattern in the rectangular region $[0,w] \times [0,h]$ having the following grid labels: we put $\varepsilon_{0}=\varepsilon_{h}=\varepsilon_{\beta-1}=\varepsilon_{\beta}=\eta_0=\eta_w=\eta_{\alpha-1}=\eta_{\alpha}=0$ and set all other grid labels equal to $1$.  This hitomezashi pattern has one loop of width $w$ and height $h$; we define the \dfn{cross} $\Cross^{\alpha,\beta}_{w\times h}$ to be this hitomezashi loop, considered modulo translation.  We remark that a cross can also be viewed as the boundary of the region formed by the union of a $\Rug_{w \times 3}$ and a $\Rug_{3 \times h}$, as seen in the middle of Figure~\ref{fig:rug}.

Now suppose $w,h \geq 5$ are odd and $h \equiv 1 \pmod{4}$, and choose an odd integer $3 \leq \alpha \leq w-2$.  Consider the hitomezashi pattern in the rectangular region $[0,w] \times [0,h]$ having the following grid labels: we put $\varepsilon_0=\varepsilon_h=0$, and for the remaining horizontal grid labels, we let $\varepsilon_i$ equal $0$ if $i\equiv 2,3 \pmod{4}$ and equal $1$ if $i \equiv 0,1 \pmod{4}$; we put $\eta_0=\eta_{\alpha-1}=\eta_{\alpha}=\eta_w=0$ and let all other vertical grid labels equal $1$.  This hitomezashi pattern has one loop of width $w$ and height $h$; we define the \dfn{horizontal comb} $\HComb^{\alpha}_{w\times h}$ to be this hitomezashi loop, considered modulo translation. The image on the right of Figure~\ref{fig:rug} shows $\HComb_{11\times 13}^7$. If $w \equiv 1 \pmod{4}$ and $3 \leq \beta \leq h-2$ is odd, then we can define the \dfn{vertical comb} $\VComb^{\beta}_{w\times h}$ analogously.  We remark that a horizontal comb can also be viewed as the boundary of the region formed by the union of a $\Rug_{3 \times h}$ and several copies of $\Rug_{w \times 3}$; a similar statement holds for vertical combs.

Finally, we introduce \dfn{wands} when $\min\{w,h\} = 5$. A wand of width $5$ has $\eta_1=\eta_4 = 1$, $\eta_0=\eta_2=\eta_3=\eta_5=0$, $\varepsilon_0 = \varepsilon_h=0$, $\varepsilon_1= \varepsilon_{h-1} = 1$, and $\varepsilon_{2i} = \varepsilon_{2i+1}$ for all $i \in [1,(h-3)/2]$, with the requirement that $\varepsilon_{2i}=0$ for at least one $i \in [1,(h-3)/2]$. A wand of height $5$ is a $90^{\circ}$ rotation of a wand of width $5$. Note that a comb with $\min\{w,h\} = 5$ is a type of wand. See Figure~\ref{fig:wand}.

\begin{figure}[htbp]
\includegraphics[height=6.482cm]{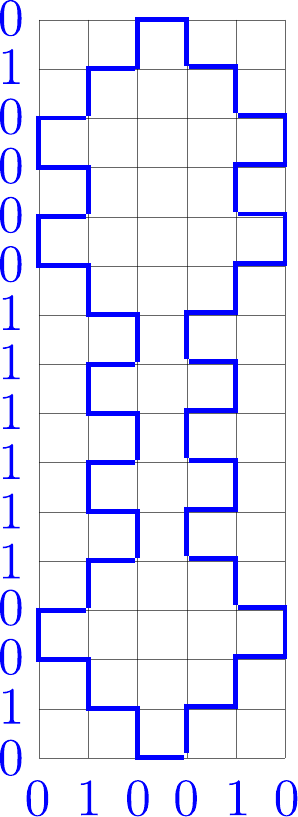}
\caption{A wand of width $5$ and height $15$.}
\label{fig:wand}
\end{figure}

We omit the easy proof of the following proposition, which states only the lengths and areas that we will need later.

\begin{proposition}\label{ex:rug length}
We have 
\begin{align*}
    \length(\Rug_{w\times h})&=4(w+h-3)\\
    \area(\Rug_{w\times h})&=(w-1)(h-1)+1,\\
    \area(\Cross_{w\times h}^{\alpha,\beta})&=2(w+h)-7,\\
    \length(\HComb_{w\times h}^\alpha)&=\length(\VComb_{w\times h}^\beta)=(w-1)(h-1)+4,
\end{align*}
and the length of a wand of width $5$ and height $h$ is $4h$.
\end{proposition}

We now proceed to our minimization problems.

\begin{theorem}
\label{thm:min-length}
Let $w,h \geq 1$ be odd integers.  If $L$ is a hitomezashi loop with width $w$ and height $h$, then 
\begin{equation*}
    \length(L) \ge 4\max\{w,h\}.
\end{equation*}
If $w \ge h$ (respectively, $w\leq h$), then equality is achieved if and only if each horizontal (respectively, vertical) stitch in $L$ has the same longitude (respectively, latitude) as exactly one other stitch in $L$. 
\end{theorem}

\begin{proof}
Without loss of generality, we may assume that $w \geq h$.  Orient the loop $L$ counterclockwise, so that each stitch receives an orientation.  Then $L$ contains at least one stitch oriented west-to-east and one stitch oriented east-to-west at each of the $w$ half-integer longitudes passing through the interior of $L$, so $L$ has at least $2w$ horizontal stitches.  Since we alternately pass through horizontal and vertical stitches as we traverse $L$, we conclude that $\length(L) \geq 4w$.

We achieve the equality $\length(L)=4w$ if and only if $L$ has exactly one stitch oriented west-to-east and exactly one stitch oriented east-to-west at each of the $w$ half-integer longitudes passing through the interior of $L$, i.e., there are exactly two horizontal stitches at each such longitude.
\end{proof}

We remark that the loops achieving equality in the above theorem (for $w \geq h$) correspond to Dyck paths of semilength $(w-1)/2$ (by reading out the sequence of north and south vertical steps as one traverses the loop counterclockwise from left-extremum to right-extremum).

We can also use Theorem~\ref{thm:Pete_extremal} to give a sharp lower bound on the area of a hitomezashi loop. In this setting, the $1\times 1$ loop is so small that it behaves differently, so we require loops to have both width and height at least $3$.

\begin{theorem}
\label{thm:min-area}
Let $w,h \geq 3$ be odd integers.  If $L$ is a hitomezashi loop with width $w$ and height $h$, then 
\begin{equation*}
    \area(L) \ge 2(w+h) - 7.
\end{equation*}
Moreover, equality is achieved if and only if either $L$ is a cross or $\min\{w,h\} = 3$.
\end{theorem}

\begin{proof}
Every loop with $\min\{w,h\}=3$ is a rug and hence has area $2(w+h)-7$ by Proposition~\ref{ex:rug length}.  We henceforth restrict our attention to the case $\min\{w,h\}\geq 5$. Without loss of generality, suppose $L$ is contained in the region  $[0,w] \times [0,h]$. Fix a south-extremal stitch $s$ and a west-extremal stitch $t$ of $L$. Let $\alpha-1/2$ (respectively, $\beta-1/2$) be the longitude (respectively, latitude) of $s$ (respectively, $t$). Let $C$ be the cross $\Cross^{\alpha,\beta}_{w\times h}$ whose unique south-extremal stitch is $s$ and whose unique west-extremal stitch is $t$. It easily follows from Theorem~\ref{thm:Pete_extremal} that the interior of $C$ must be contained in the interior of $L$.  By Proposition~\ref{ex:rug length}, $\area(C)=2(w+h)-7$. Thus, $\area(L)\geq 2(w+h)-7$, and equality holds if and only if $L$ is the cross $C$.  
\end{proof}

We now turn to maximization problems for hitomezashi loops.  We begin with length. It turns out that rugs achieve the maximum length only when the width or height is very small, so we will focus our attention of non-rug loops.

\begin{theorem}\label{thm:max-length otherwise}
Let $w,h \ge 5$ be odd integers, and let $L$ be a hitomezashi loop with width $w$ and height $h$ that is not a rug.  Then
\begin{equation*}
    \length(L) \le (w-1)(h-1) + 4.
\end{equation*}
Moreover, equality is achieved if and only if $L$ is a horizontal or vertical comb or, when $\min\{w,h\}=5$, a wand.
\end{theorem}

\begin{proof}
Note that $\length(L)$ is equal to the number of lattice points through which $L$ passes.  Without loss of generality, suppose $L$ is contained in the region  $[0,w] \times [0,h]$.  We claim that if $i_0$ is an extremal latitude for $L$ and $j_0$ is an extremal longitude for $L$, then $L$ cannot pass through any of the four lattice points $(j_0 \pm 1/2, i_0 \pm 1/2)$.  By symmetry, it suffices to show the claim for the lattice point $v:=(j_0-1/2, i_0-1/2)$.

\begin{figure}[htbp]
\includegraphics[height=3.4cm]{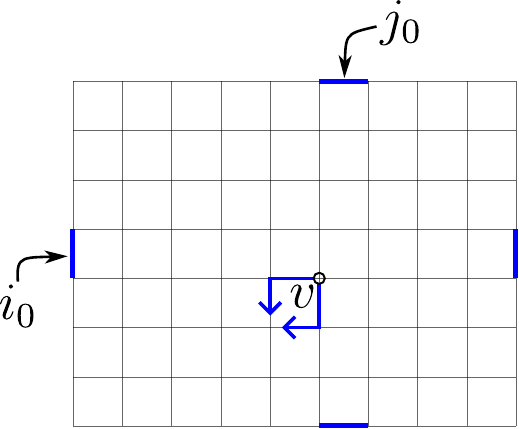}\qquad\qquad\includegraphics[height=3.4cm]{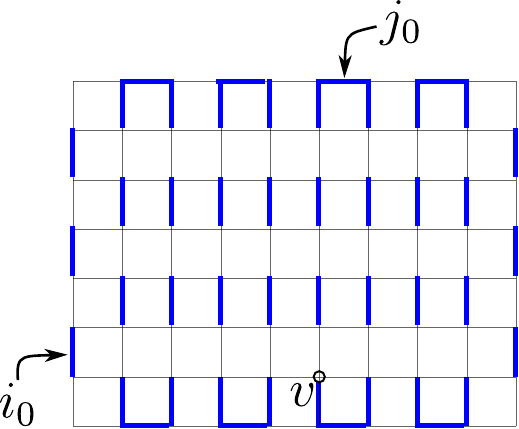}
\caption{The left image shows that we cannot have both $i_0>3/2$ and $j_0>3/2$ in the proof of Theorem~\ref{thm:max-length otherwise} because the stitches would ``close up'' into a 4-stitch loop. The right image shows how the vertical grid labels (and hence all of the vertical stitches) are determined once we assume that $i_0=3/2$.}
\label{fig:no stitches between extremal stitches}
\end{figure}

Assume for the sake of contradiction that $L$ passes through $v$.  Then, by Theorem~\ref{thm:Pete_extremal}, the stitches incident to $v$ must extend to the west and the south.  If $j_0>3/2$ and $i_0>3/2$, then two more applications of Theorem~\ref{thm:Pete_extremal} force us to ``close up'' a $4$-stitch loop passing through $v$ (see the image on the left of Figure~\ref{fig:no stitches between extremal stitches}), which contradicts our assumption that $L$ passes through $v$.  So we may (without loss of generality) restrict our attention to the case where $i_0=3/2$.  Further applications of Theorem~\ref{thm:Pete_extremal} to the extremal latitude $i_0$ determine the vertical grid labels $\eta_j=0$ for all $1 \leq j \leq w-1$, and of course we already know that $\eta_0=\eta_w=1$.  Now observe that the height of $L$ is equal to the smallest value of $i>0$ such that the horizontal grid label $\varepsilon_i$ equals $1$ (since this is when $L$ will ``close up'').  Thus, $\varepsilon_i=0$ for all $1 \leq i \leq h-1$, and of course we already know that $\varepsilon_0=\varepsilon_h=1$.  But now all of the grid labels have been determined, and we see that $L$ is a rug, contrary to our assumption.  This establishes the claim.

Now suppose that $L$ has $X$ extremal latitudes and $Y$ extremal longitudes.  Then $L$ passes through exactly
$4X+4Y$
lattice points on the boundary of the rectangular region $[0,w] \times [0,h]$.  The number of lattice points in the interior of this region through which $L$ passes is, by the claim, at most
$$(w-1)(h-1)-4XY.$$
Hence, we have the bound
$$\length(L) \leq (w-1)(h-1)-4XY+4X+4Y=(w-1)(h-1)+4-4(X-1)(Y-1).$$
Since $X,Y \geq 1$, we conclude that
$$\length(L) \leq (w-1)(h-1)+4,$$
as desired.

It remains to characterize when equality occurs.  Assume first that $\min\{w,h\} > 5$ and that $\length(L)=(w-1)(h-1)+4$.  The final inequality in the previous paragraph is tight only when $\min\{X,Y\}=1$; without loss of generality, assume that $Y=1$.  Let $j_0$ be the unique extremal longitude of $L$, and let $I$ be the set of extremal latitudes of $L$.  We see from the proof of the upper bound on $\length(L)$ that $L$ passes through all lattice points in the interior of $[0,w] \times [0,h]$ except for the points of the form $(j_0 \pm 1/2, i \pm 1/2)$ for $i \in I$.  In particular, $L$ passes through all points $(j,1)$ for $j \in [1,w-1]\setminus \{j_0 \pm 1/2\}$.  Since the only horizontal stitch of $L$ at latitude $0$ is the one at longitude $j_0$ (recall that $j_0$ is the unique extremal longitude of $L$), we determine the vertical grid labels $\eta_j=1$ for all $i \in [1,w-1]\setminus \{j_0 \pm 1/2\}$.  Of course we already know that $\eta_{j_0-1/2}=\eta_{j_0+1/2}=0$.  Note also that $j_0+1/2$ is odd since otherwise $L$ could not pass through the point $(1,1)$.  It then follows from a simple parity check that $\eta_0=\eta_w=0$. We also must have $\varepsilon_0=\varepsilon_h=0$ and $\varepsilon_1=1$. See the top left image in Figure~\ref{fig:Noah's email}.

Now that we have determined all of the vertical grid labels, we turn to the remaining horizontal grid labels.  We proceed from bottom to top, starting with $\varepsilon_2$; these steps are illustrated in Figure~\ref{fig:Noah's email}. 
\begin{enumerate}
    \item We have $\varepsilon_2=0$ since otherwise we would create a $4$-stitch loop passing through $(j,1)$ for some $j \in [1,w-1]\setminus \{j_0 \pm 1/2\}$, contradicting the fact that $L$ passes through all such lattice points. 
    \item We have $\varepsilon_3=0$ since $L$ is contained in $[0,w] \times [0,h]$.
    \item We have $\varepsilon_4=1$ since otherwise we would create a $4$-stitch loop whose center has latitude $7/2$ and longitude not equal to $j_0$. 
    \item Note that choosing $\varepsilon_5=0$ will ``close up'' $L$.  Hence, we have $\varepsilon_5=0$ if $h=5$ and $\varepsilon_5=1$ if $h>5$.
\end{enumerate}
We now iterate these four steps until we reach the latitude $h$, at which point Step (4) ``closes up'' $L$.  This determines all of the grid labels, and we recognize that $L=\HComb^{j_0+1/2}_{w \times h}$, as desired.  (If we had chosen $X=1$ instead of $Y=1$, we would have ended up with a vertical comb.)  Proposition~\ref{ex:rug length} shows that combs achieve equality.

Finally, suppose without loss of generality that $w = 5$. As before, $L$ passes through the points $(j,1)$ for $j \in [1,4] \setminus \{j_0 \pm 1/2\} = \{1,4\}$. Thus, $\eta_1 = \eta_4 = 1$, and $\eta_i = 0$ for all $i \in \{0,2,3,5\}$. Similarly, we may again conclude that $\varepsilon_0 = \varepsilon_h = 0$ and $\varepsilon_1 = \varepsilon_{h-1} = 1$. What is interesting for $w = 5$ is that the previous conclusion about $\varepsilon_2$ no longer holds: the loop $L$ already passes through both $(1,1)$ and $(4,1)$. In other words, there is simply not enough width to cause concern. In fact, the only conclusions we can draw are that: (1)~$\varepsilon_{2i} = \varepsilon_{2i+1}$ for each $i \in [1,(h-3)/2]$ so that the loop neither closes before achieving height $h$ nor has width greater than $5$; and (2)~there exists $i \in [1,(h-3)/2]$ with $\varepsilon_{2i} = \varepsilon_{2i+1} = 0$ so that the loop does achieve width $h$. In other words, the loop is a wand. 
From Proposition~\ref{ex:rug length}, we know that the length of the wand is $4h = (5-1)(h-1)+4$, completing the proof.
\end{proof}

\begin{figure}[htbp]
\[\begin{array}{l}\includegraphics[height=4.287cm]{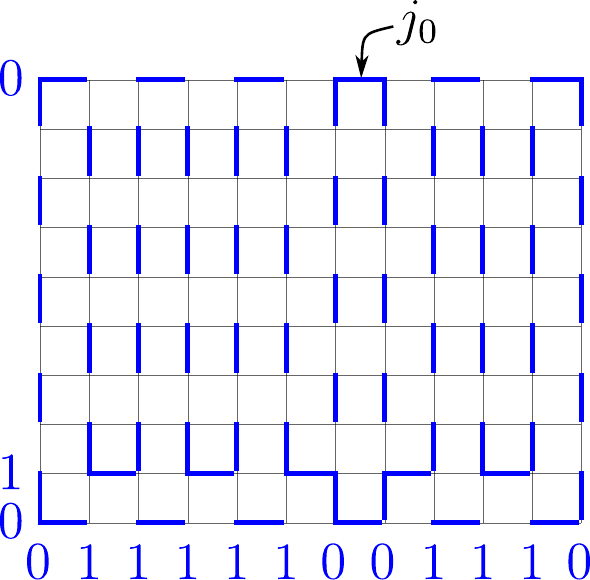}\,\,\raisebox{1.9cm}{$\xrightarrow{\,\,\,(1)\,\,\,}$}\,\,\includegraphics[height=4.287cm]{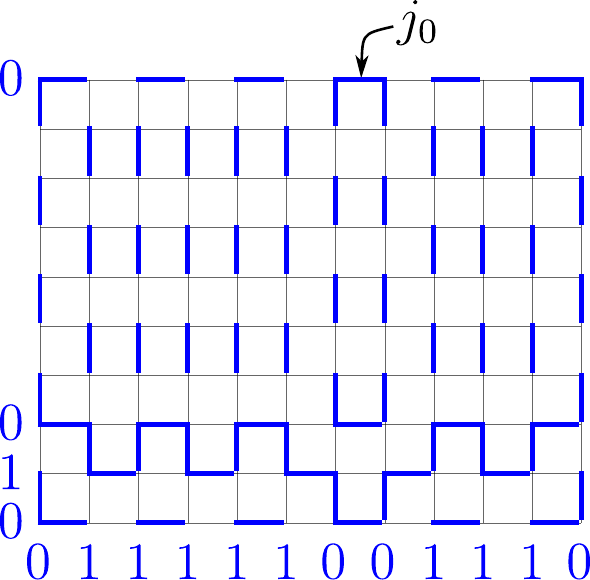}\,\,\raisebox{1.9cm}{$\xrightarrow{\,\,\,(2)\,\,\,}$}\,\,\includegraphics[height=4.287cm]{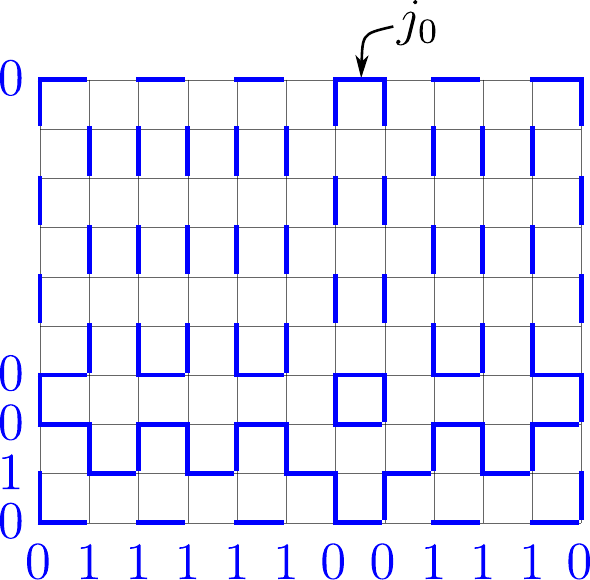}\end{array}\]
\[\begin{array}{l}\hphantom{\includegraphics[height=4.287cm]{HitoBPIC30}}\,\,\raisebox{1.9cm}{$\xrightarrow{\,\,\,(3)\,\,\,}$}\,\,\includegraphics[height=4.287cm]{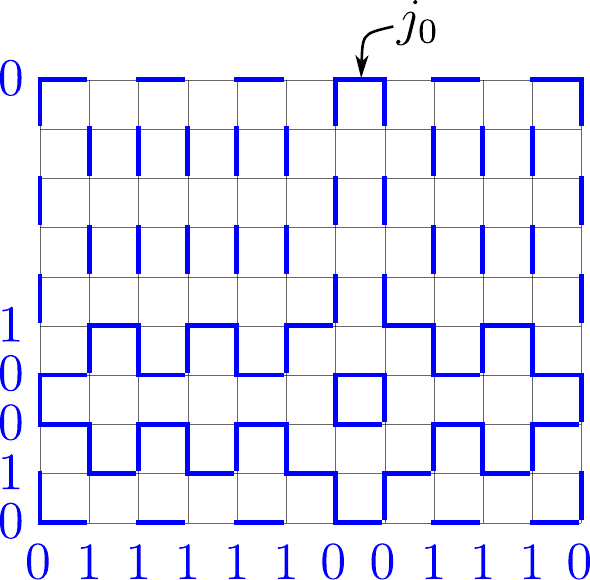}\,\,\raisebox{1.9cm}{$\xrightarrow{\,\,\,(3)\,\,\,}$}\,\,\includegraphics[height=4.287cm]{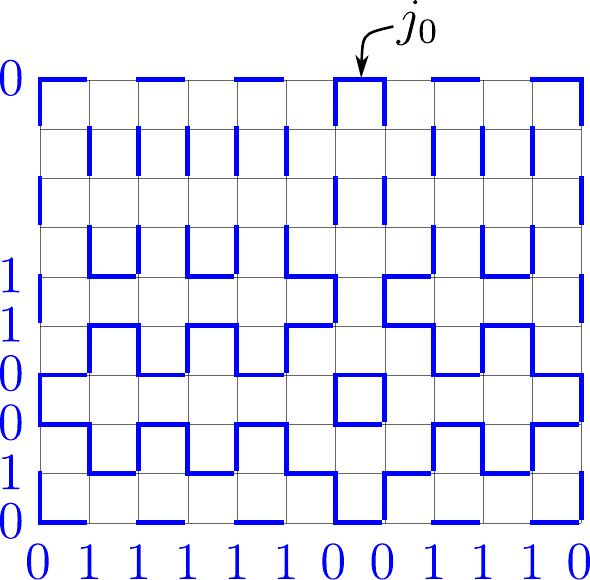}\end{array}\]
\caption{These partial hitomezashi patterns show the step-by-step determination of the grid labels relevant to our extremizing loop $L$.  In the top-left image, we have determined all of the vertical grid labels but only three of the horizontal grid labels.  The remaining images show the outcomes of applying Steps (1)--(4).}
\label{fig:Noah's email}
\end{figure}

We remark that $\length(\Rug_{w\times h})<(w-1)(h-1)+4$ if and only if $\max\{w,h\}\leq 5$ or $w+h \leq 16$.

Maximizing the area of a hitomezashi loop is notably easier than maximizing length.

\begin{theorem}
\label{thm:max-area}
Let $w,h \geq 1$ be odd integers.  If $L$ is a hitomezashi loop with width $w$ and height $h$, then 
\begin{equation*}
    \area(L) \le (w-1)(h-1) + 1.
\end{equation*}
Moreover, equality is achieved if and only if $L$ is a rug.
\end{theorem}

\begin{proof}
The loop $L$ lives in a $w\times h$ rectangular region, and the interior of $L$ must omit (at least) every other cell along the perimeter of this rectangle, including the four corners. So
$$\area(L) \leq wh-(2w+2h-4)/2=(w-1)(h-1)+1.$$
Proposition~\ref{ex:rug length} tells us that rugs achieve this area, and it is immediate that they are the only loops to do so.
\end{proof}

\section{Long stitches on the square grid}\label{sec:long-square}

In this section, we study the $(a,b,c,d)$-hitomezashi patterns defined in Section~\ref{sec:intro}. We can view the stitches in such a pattern as the edges of a graph whose vertex set is the set of endpoints of stitches. A connected component in this graph is an \dfn{$(a,b,c,d)$-strand} (or just $\dfn{strand}$ when $a,b,c,d$ are clear from context), and a bounded strand (which is graph-theoretically a cycle) is an \dfn{$\longstitch$-loop} (or just a \dfn{loop}). Thus, $(1,1,1,1)$-loops are the same as ordinary hitomezashi loops. 

A natural problem is to classify the $\longstitch$-hitomezashi patterns for different choices of the parameters $a,b,c,d$.  We start by establishing a necessary condition on these parameters for there to exist even a single $\longstitch$-hitomezashi pattern.

\begin{lemma}\label{lem:density}
If an $\longstitch$-hitomezashi pattern exists, then $a+b=c+d$. 
\end{lemma}

\begin{proof}
Suppose $H$ is an $\longstitch$-hitomezashi pattern.  Consider an $(a+b-1) \times (c+d-1)$ rectangular subset $S$ of the integer lattice.  On the one hand, by looking at each latitude individually, we find that the number of stitch endpoints in $S$ is $2(a+b-1)$.  On the other hand, by looking at each longitude individually, we find that the number of stitch endpoints in $S$ is $2(c+d-1)$.  These two quantities must be equal, so $a+b=c+d$.
\end{proof}

In light of Lemma~\ref{lem:density}, we will henceforth restrict our attention to the case where $a+b = c+d$.

Let $g$ be a positive integer. If $H$ is an $\longstitch$-hitomezashi pattern and $(x,y)\in\mathbb Z^2$, then we write $(x,y)+g \cdot H$ for the partial $(ga, gb, gc, gd)$-hitomezashi pattern that is obtained by dilating $H$ by a factor of $g$ and then shifting the result by $(x,y)$.  Suppose $H_1, \ldots, H_g$ are $\longstitch$-hitomezashi patterns, and let $\sigma$ be a permutation of the set $[g]$.  Then the set of stitches in the union of $(1,\sigma(1))+g \cdot H_1, (2,\sigma(2))+g \cdot H_2, \ldots, (g,\sigma(g))+g \cdot H_g$
is the set of stitches of a $(ga, gb, gc, gd)$-hitomezashi pattern. Informally, this says that we can obtain a $(ga,gb,gc,gd)$-hitomezashi pattern by dilating each of $H_1,\ldots,H_g$ by a factor of $g$ and then overlaying these dilated hitomezashi patterns so that their relative positions are determined by $\sigma$. See Figure~\ref{FigHitoB6} for an example. It is straightforward to check that every $(ga,gb,gc,gd)$-hitomezashi pattern arises in this way. Thus, in order to complete our classification of $\longstitch$-hitomezashi patterns, it suffices to study the case
$$\gcd(a,b,c,d)=1.$$
Note that this does not force $a,b,c,d$ to be \emph{pairwise} coprime. We will also assume that $(a,b,c,d)\neq (1,1,1,1)$ since the combinatorial properties of ordinary hitomezashi patterns are quite different; as mentioned above, every assignment of grid labels results in a valid $(1,1,1,1)$-hitomezashi pattern, and thus no further classification is possible.

\begin{figure}[ht]
  \begin{center}\includegraphics[height=6.786cm]{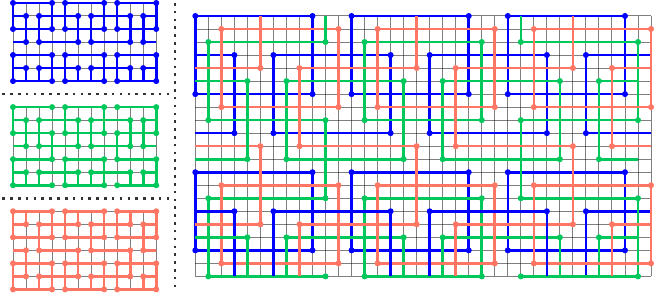}
  \end{center}
  \caption{On the left are (finite portions of) three $(3,1,2,2)$-hitomezashi patterns. On the right, we have dilated these patterns by a factor of $3$ and overlaid them (with suitable shifts) to form a $(9,3,6,6)$-hitomezashi pattern.}\label{FigHitoB6}
\end{figure}

We are now ready to introduce the types of strands that appear in $\longstitch$-hitomezashi patterns.
\begin{itemize}
    \item A loop consisting of four stitches is called a \dfn{rectangle}.
    \item An infinite strand that can be oriented so that stitches are alternately oriented west-to-east and south-to-north is called a \dfn{positive zig-zag}. An infinite strand that can be oriented so that stitches are alternately oriented west-to-east and north-to-south is called a \dfn{negative zig-zag}. The \dfn{type} of such a strand is its positive/negative direction, and a strand of either type is simply a \dfn{zig-zag}.
    \item An infinite strand that can be oriented so that stitches are alternately oriented west-to-east, south-to-north, east-to-west, and south-to-north is called a \dfn{vertical accordion}.  Similarly, an infinite strand that can be oriented so that stitches are alternately oriented south-to-north, west-to-east, north-to-south, and west-to-east is called a \dfn{horizontal accordion}.  A strand of either type is an \dfn{accordion}.
\end{itemize}
See Figure~\ref{FigHitoB5} for examples of rectangles, zig-zags, and accordions.

\begin{figure}[ht]
  \begin{center}\includegraphics[height=3cm]{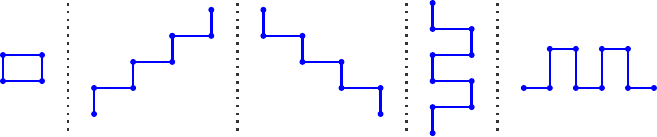}
  \end{center}
  \caption{On the far left is a rectangle. The other images, from left to right, are finite portions of a positive zig-zag, a negative zig-zag, a vertical accordion, and a horizontal accordion.}\label{FigHitoB5}
\end{figure}

In the following subsections, we will characterize the quadruples $\longstitch$ that admit rectangles, zig-zags, and accordions, and we will study how these classes of strands can fit together in $(a,b,c,d)$-hitomezashi patterns. Combining these analyses will give our main structure theorem.

As mentioned earlier, the cases where $a=b$ or $c=d$ are particularly complicated because they admit additional flexibility.  To see this, suppose we know that there are vertical stitches with endpoints at $(0,0)$ and $(a,0)$.  If $a \neq b$, then we can immediately conclude that the pattern contains the horizontal stitch $[(0,0),(a,0)]$.  If instead $a=b$, then it is also possible that the pattern instead contains the horizontal stitches $[(-a,0),(0,0)]$ and $[(a,0),(2a,0)]$.

\subsection{Rectangles}\label{sec:rectangle}
We begin by studying patterns in which all strands are rectangles.  Our first lemma characterizes the quadruples $\longstitch$ for which such an $\longstitch$-hitomezashi pattern is possible.

\begin{lemma}\label{lem:rectangle-quadruples}
Suppose that $\longstitch \neq (1,1,1,1)$ and $\gcd(a,b,c,d) = 1$, and set $M := a+b = c+d$. Suppose that $H$ is an $\longstitch$-hitomezashi pattern such that every strand is a rectangle.  Then $a$, $b$, $c$, and $d$ each have even (additive) order modulo $M$.
\end{lemma}

\begin{proof}
By symmetry, it suffices to show that $a$ has even order modulo $M$.  This conclusion clearly holds if $a=b$, so we may restrict our attention to the case $a \neq b$.  We claim that if $R$ is a rectangle in $H$, then the vertical shift $(0,M)+R$ is also a rectangle in $H$.  Without loss of generality, suppose $R$ is the rectangle whose southwest corner is at the origin.  Then $H$ contains vertical stitches $[(0,M),(0,M+c)]$ and $[(a,M),(a,M+c)]$.  Since $a \neq b$, $H$ must also have the horizontal stitches $[(0,M),(a,M)]$ and $[(0,M+c),(a,M+c)]$, as desired.  The same argument shows that $(0,-M)+R$ is a rectangle in $H$.  Say that a longitude is a \dfn{left-longitude} (respectively, \dfn{right-longitude}) if every vertical stitch at that longitude is the left (respectively, right) side of a rectangle.  It follows from the claim that each longitude is either a left-longitude or a right-longitude.

Observe that if $j$ is a left-longitude, then $j+M$ is also a left-longitude.  Indeed, if $j$ is a left-longitude, then there is a horizontal stitch $s$ whose west endpoint has longitude $j$; the west endpoint of the stitch $(M,0)+s$ 
has longitude $j+M$ and hence $j+M$ is a left-longitude. Let $q$ be the order of $a$ modulo $M$, and assume without loss of generality that $0$ is a left-longitude. On the one hand, by the preceding observation, $qa$ is also a left-longitude. 
On the other hand, since $0$ is a left-longitude, $a$ must be a right-longitude. This then forces $2a$ to be a left-longitude (otherwise, $a$ would be both a left- and a right-longitude, which is impossible), which forces $3a$ to be a right-longitude, and so on. Since $qa$ is a left-longitude, we see that $q$ must be even.
\end{proof}

For the remainder of this subsection, we will focus on the case where $a \neq b$ and $c \neq d$.  Our task is to characterize the $\longstitch$-hitomezashi patterns in which all strands are rectangles.  We first set up an explicit parameterization of these patterns.  Let $\longstitch \neq (1,1,1,1)$ and $\gcd(a,b,c,d) = 1$, and set $M := a+b = c+d$. Furthermore, following Lemma~\ref{lem:rectangle-quadruples}, assume that $a \neq b$ and $c \neq d$ and that $a$ (respectively, $c$) has even order $q$ (respectively, $r$) modulo $M$. Write $\{0,1\}^s$ for the set of $\{0,1\}$-sequences of length $s$, and write $S_n$ for the set of permutations of the set $[n]$. 
We define a map $\varphi$ (depending implicitly on $\longstitch$) from $$\mathcal{X}_{\longstitch}:=\{0,1\}^{M/q} \times \{0,1\}^{M/r} \times S_{M/2}$$ to the set of $\longstitch$-hitomezashi patterns in which all strands are rectangles, as follows.  Consider an element
$$\vec{x}=((u_1, \ldots, u_{M/q}), (v_1, \ldots, v_{M/r}), \sigma)\in \mathcal{X}_{\longstitch}.$$
For each $i$, consider the sequence
$$\hat{u}(i)=i+u_i a \, , \ i+(u_i+2)a \, , \ i+(u_i+4)a \, , \ \ldots, \ i+(u_i+q-2)a \, ,$$
and let $\widetilde{u}$ be the concatenation $$\widetilde{u}=\hat{u}(1)\hat{u}(2) \cdots \hat{u}(M/q):=\widetilde{u}_1 \cdots \widetilde{u}_{M/2}$$
(where $\widetilde{u}_1, \ldots, \widetilde{u}_{M/2}$ are individual elements).  Define sequences $\hat{v}_j$ and $\widetilde{v}$ analogously.  Finally, let $\varphi(\vec{x})$ be the $\longstitch$-hitomezashi pattern $H$ consisting entirely of rectangles, where the set of southwest corners of rectangles is
$$\{(\widetilde{u}_1, \widetilde{v}_{\sigma(1)}), (\widetilde{u}_2, \widetilde{v}_{\sigma(2)}),\ldots, (\widetilde{u}_{M/2}, \widetilde{v}_{\sigma(M/2)}) \}+M \cdot \mathbb{Z}^2.$$
The content of the following lemma is that, via $\varphi$, the set $\mathcal{X}_{\longstitch}$ in fact parameterizes the $(a,b,c,d)$-hitomezashi patterns consisting entirely of rectangles.

\begin{lemma}\label{lem:rectangle-characterization}
Let $\longstitch \neq (1,1,1,1)$ and $\gcd(a,b,c,d) = 1$, and set $M := a+b = c+d$. Furthermore, assume that $a \neq b$ and $c \neq d$, and that $a$ (respectively, $c$) has even order $q$ (respectively, $r$) modulo $M$. Then the map $\varphi$ is a bijection between $\mathcal{X}_{\longstitch}$ and the set of $\longstitch$-hitomezashi patterns in which every strand is a rectangle.
\end{lemma}

\begin{proof}
It is immediate that every $\varphi(\vec{x})$ is an $\longstitch$-hitomezashi pattern consisting entirely of rectangles and that $\varphi$ is injective.  It remains to show that $\varphi$ is surjective.  Let $H$ be an $\longstitch$-hitomezashi pattern consisting entirely of rectangles.  The argument from the proof of Lemma~\ref{lem:rectangle-quadruples} (first paragraph) shows that if $R$ is a rectangle in $H$, then $(\pm M,0)+R$ and $(0,\pm M)+R$ are also rectangles in $H$.  (This is the only place where we make essential use of the assumption that $a \neq b$ and $c \neq d$.)  It follows that the grid label $\varepsilon_i$ (respectively, $\eta_j$) depends only on the residue class of $i$ (respectively, $j$) modulo $M$.  We now focus on the behavior of $H$ ``modulo $M$,'' i.e., in an $M \times M$ box.

For each $1 \leq i \leq M/q$, consider the rectangles of $H$ whose vertical sides have longitudes equivalent to $i$ modulo $M/q$.  The set of longitudes (modulo $M$) of the west sides of these rectangles is either $\{i, i+2a, \ldots, i+(q-2)a\}$ or $\{i+a, i+3a, \ldots, i+(q-1)a\}$, depending on whether the vertical stitches at longitude $i$ are the left or right sides of their rectangles.  Set $u_i=0$ in the first case and $u_i=1$ in the second case.  For each $1 \leq j \leq M/r$, produce $v_j \in \{0,1\}$ in the same way, i.e., according to whether the horizontal stitches at latitude $j$ are the bottom or top sides of their rectangles.

From these sequences $u_1, \ldots, u_{M/q}$ and $v_1, \ldots, v_{M/r}$ , consider the sequences $\widetilde{u}$ and $\widetilde{v}$ (as defined in the set-up for the lemma).  For each $1 \leq k \leq M/2$, there is a unique $1 \leq \ell \leq M/2$ such that the point $(\widetilde{u}_k, \widetilde{v}_\ell)$ is the southwest corner of a rectangle of $H$; define the permutation $\sigma \in S_{M/2}$ by setting $\sigma(k)=\ell$ for each $k$.  At last, we see that
$$\varphi((u_1, \ldots, u_{M/q}), (v_1, \ldots, v_{M/r}), \sigma)$$
agrees with $H$ on all rectangles with southwest corners in the box $[1,M] \times [1,M]$, and we conclude by the $M \cdot \mathbb{Z}^2$-periodicity of both $H$ and $\varphi((u_1, \ldots, u_{M/q}), (v_1, \ldots, v_{M/r}), \sigma)$ that in fact these two patterns are equal, as desired.
\end{proof}

\subsection{Zig-zags}\label{sec:zig-zag}
Zig-zags turn out to be the most rigid of the three types of strand.

\begin{lemma}\label{lem:zig-zag}
Let $\longstitch \neq (1,1,1,1)$ and $\gcd(a,b,c,d) = 1$, and set $M := a+b = c+d$.
\begin{enumerate}
    \item Suppose that $H$ is an $\longstitch$-hitomezashi pattern with a zig-zag $Z$.  Then $\gcd(a,b)=\gcd(c,d)=1$.  Moreover, $H$ is the unique $\longstitch$-hitomezashi pattern containing $Z$, and every strand in $H$ is a zig-zag of the same type as $Z$.
    \item If $\gcd(a,b)=\gcd(c,d)=1$, then there exists an $\longstitch$-hitomezashi pattern $H$ as described in part (1).
\end{enumerate}  
\end{lemma}

\begin{proof}
We begin with part (1).  Without loss of generality, we can assume that $Z$ is a positive zig-zag containing the horizontal stitch $[(0,0),(a,0)]$.  We claim that $a \neq b$.  Indeed, if instead $a=b$, then $H$ contains the horizontal stitch $[(2a,0),(3a,0)]$, and so $(2a,0)$ is the endpoint of a vertical stitch of $H$.  This vertical stitch must be $[(2a,0),(2a, -c)]$ since we already know that $Z$ contains the vertical stitch $[(2a,c),(2a,2c)]$.  The difference in latitude between the south endpoints of these two vertical stitches, namely, $2c$, must be a multiple of $M$.  So $2c=M$, i.e., $a=b=c=d$.  The condition $\gcd(a,b,c,d)=1$ implies that $a=b=c=d=1$, a contradiction.  In the same way, we find that $c \neq d$.

Note that $aM/\gcd(a,b)$ is an integer multiple of $M$.  Consider the horizontal stitches \[[(aM/\gcd(a,b),0),(aM/\gcd(a,b)+a,0)]\quad \text{and}\quad[(aM/\gcd(a,b)-a,-c),(aM/\gcd(a,b),-c)],\] which are forced to be in $H$ by the existence of $Z$.  Since $c \neq d$, the vertical stitch incident to $(aM/\gcd(a,b),0)$ must have its other endpoint at $(aM/\gcd(a,b),-c)$.  We also know that $Z$ contains the vertical stitch $[(aM/\gcd(a,b),cM/\gcd(a,b)),(aM/\gcd(a,b),cM/\gcd(a,b)-c)]$.  See Figure~\ref{fig:zig-zag}. Hence, $cM/\gcd(a,b)$ is a multiple of $M$. Thus, $c$ (and hence also $d$) is a multiple of $\gcd(a,b)$.  It follows that $\gcd(a,b)=\gcd(a,b,c,d)=1$.  Likewise, $\gcd(c,d)=1$.

We now prove the second statement of part (1).  Note that the stitches of $Z$ determine all of the horizontal grid labels at latitudes that are multiples of $c$.  Following the arguments of the previous paragraph (crucially using the fact that $a \neq b$ and $c \neq d$), we see that every shift of $Z$ by an element of  $M \cdot \mathbb{Z}^2$ is also a strand of $H$.  Since $c$ is coprime to $M$ (indeed, $\gcd(c,M)=\gcd(c,d)=1$), these shifts of $Z$ determine all of the remaining horizontal grid labels.  The vertical grid labels are uniquely determined by $Z$ in the same way, and the uniqueness of $H$ follows.

To see that all other strands in $H$ are also positive zig-zags, we simply read off the grid labels as follows: The horizontal grid label at latitude $i$ is $\varepsilon_i=iac^{-1} \pmod{M}$, and the vertical grid label at longitude $j$ is $\eta_j=jc a^{-1}-c \pmod{M}$.  (Here, $a^{-1}, c^{-1}$ denote the inverses of $a,c$ modulo $M$.)

Finally, for part (2) of the lemma, one can easily check that the grid labels described in the previous paragraph indeed result in a valid $(a,b,c,d)$-hitomezashi pattern (i.e., the endpoints of the vertical stitches coincide with the endpoints of the horizontal stitches).
\end{proof}

\begin{figure}[ht]
  \begin{center}\includegraphics[height=4.95cm]{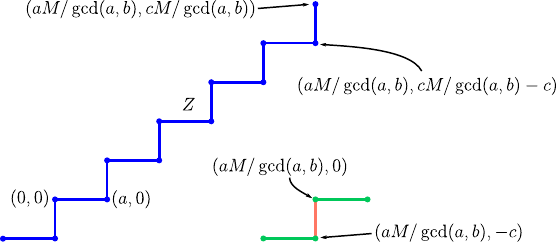}
  \end{center}
  \caption{An illustration of the proof of Lemma~\ref{lem:zig-zag}. The existence of the zig-zag $Z$ in $H$ forces the two green stitches to also be in $H$, and these force the orange stitch to be in $H$. }\label{fig:zig-zag}
\end{figure}

\begin{remark}\label{rem:zig-zag}
As a consequence of part (1) of Lemma~\ref{lem:zig-zag}, zig-zags are impossible when $a=b$ or $c=d$ (unless $(a,b,c,d)=(1,1,1,1)$). 
Also, a pattern of zig-zags as in the lemma is $M$-bi-periodic; i.e., invariant under translation by elements of $M \cdot \mathbb{Z}^2$.
\end{remark}

We conclude this subsection by showing that the presence of just three consecutive zig-zagging stitches forces an entire strand to be a zig-zag. Notice that this lemma dovetails nicely with the ``moreover'' statement in part (1) of Lemma~\ref{lem:zig-zag}.

\begin{lemma}\label{lem:3-stitches}
Suppose that $\longstitch \neq (1,1,1,1)$, $\gcd(a,b,c,d) = 1$, $a \neq b$, and $c \neq d$, and set $M := a+b = c+d$. Suppose that $H$ is an $\longstitch$-hitomezashi pattern that contains the vertical stitches $[(0,-c),(0,0)]$ and $[(a,0),(a,c)]$ and the horizontal stitch $[(0,0),(a,0)]$.  Then the strand in $H$ containing these three stitches is a positive zig-zag.
\end{lemma}

\begin{proof}
The horizontal stitch incident to the point $(a,c)$ must connect to $(2a,c)$ (rather than to $(0,c)$).  Likewise, the vertical stitch incident to $(2a,c)$ must connect to $(2a,2c)$ (rather than to $(2a,0)$), and so on.  The same argument allows us to extend the partial zig-zag in the other direction.
\end{proof}

\subsection{Putting everything together in the generic case}\label{sec:together}

We are now ready to complete our classification in the ``generic case'' where $a \neq b$ and $c \neq d$.  The following lemma gives our key dichotomy.

\begin{lemma}\label{lem:dichotomy}
Let $(a,b,c,d)$ be such that $\gcd(a,b,c,d)=1$, $a \neq b$, $c \neq d$, and set $M := a+b=c+d$.  Let $H$ be an $\longstitch$-hitomezashi pattern.  Then either all of the strands in $H$ are rectangles or all of the strands are zig-zags of the same type. 
\end{lemma}

\begin{proof}
If every strand in $H$ is a rectangle, then we are done, so consider the case where there is a strand $S$ that is not a rectangle. Then, without loss of generality  (that is, after a translation and possibly a reflection), this $S$ contains three consecutive stitches that are arranged as in the hypothesis of Lemma~\ref{lem:3-stitches}; the conclusion  of that result implies that $S$ is a zig-zag.  Then part (1) of Lemma~\ref{lem:zig-zag} tells us that all of the strands in $H$ are zig-zags having the same type as $S$.
\end{proof}

We now complete our classification theorem, including the enumeration of $\longstitch$-hitomezashi patterns, for the generic case.

\begin{theorem}[Classification theorem, generic case]\label{thm:classification-generic}
Let $(a,b,c,d)$ be such that $\gcd(a,b,c,d)=1$, $a \neq b$, $c \neq d$, with $M:= a+b=c+d$, and $a$ and $c$ have orders $q$ and $r$, respectively, modulo $M$.
\begin{enumerate}
    \item If $\gcd(a,b) \neq \gcd(c,d)$ and at least one of $q,r$ is odd, then there are no $\longstitch$-hitomezashi patterns.
    \item If $\gcd(a,b) \neq \gcd(c,d)$ and $q,r$ are both even, then every $\longstitch$-hitomezashi pattern consists entirely of rectangles (as described in Lemma~\ref{lem:rectangle-characterization}), and the number of such patterns is $2^{M/q+M/r}\cdot (M/2)!$.
    \item If $\gcd(a,b)=\gcd(c,d)$ and at least one of $q,r$ is odd, then every $\longstitch$-hitomezashi pattern consists entirely of zig-zags of a single type, and the number of such patterns is $2M$.
    \item If $\gcd(a,b)=\gcd(c,d)$ and $q,r$ are both even, then every $\longstitch$-hitomezashi pattern consists either entirely of rectangles (as described in Lemma~\ref{lem:rectangle-characterization}) or entirely of zig-zags of a single type, and the total number of such patterns is $2^{M/q+M/r}\cdot (M/2)!+2M$.
\end{enumerate}
\end{theorem}

\begin{proof}
The classification result follows immediately from Lemmas~\ref{lem:dichotomy} (dichotomy), \ref{lem:rectangle-quadruples} (quadruples allowing rectangles), \ref{lem:rectangle-characterization} (characterization of rectangles), and~\ref{lem:zig-zag} (characterization  of zig-zags).  To enumerate patterns consisting of rectangles, note that
$$|\mathcal{X}_{\longstitch}|=2^{M/q+M/r} \cdot (M/2)!.$$
To enumerate patterns consisting of zig-zags, note that there are $M$ choices for the horizontal grid label $\varepsilon_0$ and that, for each such choice, there is a unique pattern consisting of positive zig-zags and a unique pattern consisting of negative zig-zags.
\end{proof}

\subsection{The $a=b$ case: rectangles and accordions}\label{sec:accordion}

In this last subsection, we treat the ``non-generic'' cases where $a=b$ or $c=d$.  Since these two cases are identical (just rotate the entire pattern), we will work with the case where $a=b$, which implies that $c \neq d$ by our standing assumptions that $\gcd(a,b,c,d)=1$ and $(a,b,c,d)\neq (1,1,1,1)$.  It turns out that now a single $\longstitch$-hitomezashi pattern can contain strands of different types, so we give a unified treatment.  Our first lemma is analogous to Lemma ~\ref{lem:rectangle-quadruples} from Section~\ref{sec:rectangle}.

\begin{lemma}\label{lem:a=b-triples}
Suppose that $(a,c,d) \neq (1,1,1)$ and $\gcd(a,c,d) = 1$, and set $M := 2a = c+d$ (so $c \neq d$).  Let $H$ be an $(a,a,c,d)$-hitomezashi pattern.  Then the vertical grid labels satisfy the relation $\eta_j=\eta_{j+a}$ for all $j \in \mathbb{Z}$.  Moreover, each strand in $H$ is either a rectangle or a horizontal accordion, and $c$ (equivalently, $d$) has even order modulo $M$.
\end{lemma}

\begin{proof}
To see the first claim, suppose that $\eta_0=0$, i.e., there is a vertical stitch from $(0,0)$ to $(0,c)$.  Then $(a,0)$ and $(a,c)$ are the endpoints of horizontal stitches.  Since $c \neq d$, we see that there must also be a vertical stitch from $(a,0)$ to $(a,c)$, i.e., we also have $\eta_a=0$.  By shift-invariance, this establishes the first part of the lemma.

Let us continue the computation from the previous paragraph.  Note that $\varepsilon_0, \varepsilon_c \in \{0,a\}$.  If $\varepsilon_0=\varepsilon_c$, then all of the strands of $H$ involving horizontal stitches at the latitudes $0$ and $c$ are rectangles.  If instead $\varepsilon_0 \neq\varepsilon_c$, then there is a single horizontal accordion that passes through all of the horizontal stitches at the latitudes $0$ and $c$.  This establishes the second part of the lemma.

For the third part of the lemma, note that we can define top-latitudes and bottom-latitudes as in the proof of Lemma~\ref{lem:rectangle-quadruples} and likewise argue that $c$ has even order modulo $M$.
\end{proof}

We now set up notation for an explicit parameterization of the possible $(a,a,c,d)$-hitomezashi patterns.  Let $a,c,d$ be positive integers, not all equal to $1$, such that $\gcd(a,c,d)=1$, $2a=c+d=M$ (so $c \neq d$), and $c$ has even order $q$ modulo $M$.  We define a map $\psi$ (depending implicitly on $a,c,d$) from
$$\mathcal{Y}_{(a,c,d)}:=\{0,1\}^{M/q}\times \{0,1\}^\mathbb{Z} \times S_{M/2}$$
to the set of $(a,a,c,d)$-hitomezashi patterns, as follows.  (Here, $\{0,1\}^\mathbb{Z}$ is the set of all bi-infinite $\{0,1\}$-sequences, with positions indexed by the integers.)  Consider an element
$$\vec{y}=((u_1, \ldots, u_{M/q}), (\ldots, v_{-1}, v_0, v_1, \ldots), \sigma) \in \mathcal{Y}_{(a,c,d)}.$$
Define $\widetilde{u}$ as in the discussion preceding Lemma~\ref{lem:rectangle-characterization} (using $c$ in place of $a$).  Finally, let $\psi(\vec{y})$ be the $(a,a,c,d)$-hitomezashi pattern $H$ defined as follows:
\begin{itemize}
    \item The set of south endpoints of vertical stitches is $$\{(1, \widetilde{u}_{\sigma(1)}), (2, \widetilde{u}_{\sigma(2)}), \ldots, (a, \widetilde{u}_{\sigma(a)})\}+(a\mathbb{Z} \times M\mathbb{Z}).$$ Equivalently, the vertical grid label $\eta_j$ is $\widetilde{u}_{\sigma(\overline{j})} \pmod{M}$, where $\overline{j}$ is the reduction of $j$ modulo $a$.  Note that this step completely determines the set of points appearing as the endpoints of stitches.
    \item For each latitude $i$, there are two possibilities for placing the horizontal stitches (consistent with the placement of the vertical stitches).  One of these possibilities creates a horizontal stitch whose west endpoint is in the interval $[0,a-1]$, and the other possibility creates a horizontal stitch whose west endpoint is in the interval $[a,M-1]$.  Choose the former if $v_i=0$ and the latter if $v_i=1$.  
\end{itemize}

The following lemma, which is analogous to Lemma~\ref{lem:rectangle-characterization}, shows that, via $\psi,$ the set $\mathcal{Y}_{(a,c,d)}$ in fact parameterizes the $(a,a,c,d)$-hitomezashi patterns.

\begin{lemma}\label{lem:a=b-parameterization}
Let $a,c,d$ be positive integers, not all equal to $1$, such that $\gcd(a,c,d)=1$, $2a=c+d=M$ (so $c \neq d)$, and $c$ has even order $q$ modulo $M$.  Then the map $\psi$ is a bijection between $\mathcal{Y}_{(a,c,d)}$ and the set of $(a,a,c,d)$-hitomezashi patterns.
\end{lemma}

\begin{proof}
It is immediate that every $\psi(\vec{y})$ is an $(a,a,c,d)$-hitomezashi pattern and that $\psi$ is injective.  It remains to show that $\psi$ is surjective.  Let $H$ be an $(a,a,c,d)$-hitomezashi pattern.  For each $1 \leq j \leq M/q$, consider the horizontal stitches at latitudes equivalent to $j$ modulo $M/q$.  The set of latitudes (modulo $M$) of the south endpoints of these stitches is either $\{j, j+2c, \ldots, j+(q-2)c\}$ or $\{j+c, j+3c, \ldots, j+(q-1)c\}$.  Set $u_j=0$ in the first case and $u_j=1$ in the second case, and consider the sequence $\widetilde{u}$ (as defined in the set-up for the lemma).  For each $1 \leq k \leq a$, there is a unique $\sigma(k) \in [M/2]$ such that $(k, \widetilde{u}_{\sigma(k)})$ is the south endpoint of a vertical stitch of $H$; this defines the permutation $\sigma \in S_{M/2}$.  Finally, for each $i \in \mathbb{Z}$, set $v_i=0$ if $\varepsilon_i \in [0,a-1]$, and set $v_i=1$ if instead $\varepsilon_i \in [a,M-1]$.  We see that
$$\psi((u_1, \ldots, u_{M/q}), (\ldots, v_{-1}, v_0, v_1, \ldots), \sigma)$$
agrees with $H$ on the strip $[1,a] \times \mathbb{Z}$.  By the first part of Lemma~\ref{lem:a=b-triples}, we conclude that in fact these two patterns agree everywhere.
\end{proof}

We now re-cast this lemma as a classification theorem---the $a=b$ counterpart of Theorem~\ref{thm:classification-generic}. See Figure~\ref{fig: example figure} for a finite piece of a $(2,2,3,1)$-hitomezashi pattern described by this classification. 

\begin{theorem}[Classification theorem, $a=b$ case]\label{thm:classification-a=b}
Let $a,c,d$ be positive integers, not all equal to $1$, such that $\gcd(a,c,d)=1$, with $M:=2a=c+d$ (so $c \neq d)$, and $c$ has order $q$ modulo $M$.
\begin{enumerate}
    \item If $q$ is odd, then there are no $(a,a,c,d)$-hitomezashi patterns.
    \item If $q$ is even, then every $(a,a,c,d)$-hitomezashi pattern consists of a combination of rectangles and horizontal accordions (as described in Lemma~\ref{lem:a=b-parameterization}), and the set of all such patterns has cardinality equal to the cardinality $\aleph_1$ of the continuum.
\end{enumerate}
\end{theorem}

\begin{proof}
The classification result follows immediately from Lemma~\ref{lem:a=b-parameterization} (characterization of patterns) and the second part of Lemma~\ref{lem:a=b-triples} (characterization of strands).  To enumerate patterns, note that
\[|\mathcal{Y}_{(a,c,d)}|=2^{M/q+\aleph_0} \cdot (M/2)!=\aleph_1. \qedhere\]
\end{proof}

\begin{figure}[ht]
  \begin{center}\includegraphics[height=4.51cm]{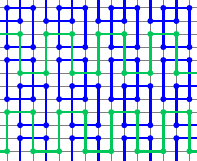}
  \end{center}
  \caption{A finite portion of a $(2,2,3,1)$-hitomezashi pattern. Rectangles appear in blue, while horizontal accordions appear in green. }\label{fig: example figure}
\end{figure}

\section{Long stitches on the triangular grid}\label{sec:long-triangular}

In this section, we briefly study the long-stitch variation of hitomezashi patterns on the triangular grid.  Although ordinary hitomezashi on the triangular grid is quite mysterious and difficult to analyze (see Section~\ref{sec:conclusion-triangular}), it turns out that the long-stitch version is much more rigid.  We define an \dfn{$(a,b)$-triangular hitomezashi pattern} to be a pattern of stitches on the triangular grid in which the stitches follow an $a$-over-$b$-under pattern in each of the three directions and the set of endpoints of stitches is the same for each of the three directions.  It is also possible to make an ``asymmetric'' version of this definition in which the different directions have stitches of different lengths, but we leave studying this generalization as an open problem.

As in the case of $(a,b,c,d)$-hitomezashi patterns (for the square grid), we can restrict our attention to the $(a,b)$-triangular hitomezashi patterns where $\gcd(a,b)=1$.  We will have three directions of stitches, at $60^{\circ}$/$120^{\circ}$ angles to each other, with one of them being horizontal. By a slight abuse of terminology, we will refer to these three directions in the triangular grid as east/west, northeast/southwest, and northwest/southeast.  Our first lemma shows that if two stitches meet at a $120^\circ$ angle, then these stitches can be ``extended'' to an entire six-stitch hexagon.

\begin{lemma}\label{lem:hexagon-completion}
Let $a,b\geq 1$ be integers, not both equal to $1$, such that $\gcd(a,b)=1$, and let $H$ be an $(a,b)$-triangular hitomezashi pattern.  Suppose $H$ contains two stitches $s_1,s_2$ that meet at a $120^\circ$ angle.  Then $H$ contains a six-stitch regular hexagon in which $s_1, s_2$ appear as consecutive sides.
\end{lemma}

\begin{proof}
Without loss of generality, we may assume that $s_1=[x,y]$ and $s_2=[y,z]$, where $y$ is northwest of $x$ and $z$ is northeast of $y$.  We claim that the east/west stitch incident to $z$ must extend to the east.  Indeed, if instead the east/west stitch incident to  $z$ extends to the west, then the west endpoint $w$ of this stitch lies a distance $a$ to the northwest of $y$; since $a \neq b$, this implies that the northwest/southeast stitch incident to $y$ has its other endpoint at $w$, but this contradicts our assumption that $s_1$ is the northwest/southeast stitch incident to $y$. 

The claim lets us ``extend'' the path with edges $s_1,s_2$ by the east/west stitch $s_3$ whose west endpoint is $z$.  By iterating this procedure, we obtain a regular hexagon with edges $s_1,s_2,s_3,s_4,s_5,s_6$, as desired.
\end{proof}

Note that exchanging the roles of $a,b$ corresponds to flipping over the cloth on which we are sewing; hence, we can restrict our attention to the case where $a>b$.  We now derive further structure from the presence of a single hexagon.

\begin{lemma}\label{lem:a=2b}
Let $a>b$ be natural numbers such that $\gcd(a,b)=1$, and let $H$ be an $(a,b)$-triangular hitomezashi pattern.  Then $a=2$ and $b=1$, and all pairs of incident stitches meet at $120^\circ$ angles.
\end{lemma}

\begin{proof}
Fix a point $x$ that is an endpoint of stitches of $H$.  By a finite case check, we find that there must be two stitches incident to $x$ that meet at a $120^\circ$ angle.  By Lemma~\ref{lem:hexagon-completion}, these two stitches are part of a six-stitch hexagon; let $y$ be the vertex of the hexagon directly opposite $x$.  Without loss of generality, we may assume that $y$ lies due east of $x$.  Note that the distance between $x$ and $y$ is $2a$.  Since $2a>a+b$ and $a+b$ is the total length of a stitch and a non-stitch, we conclude this distance $2a$ must equal either $2a+b$ (the length of a stitch followed by a non-stitch followed by a stitch) or $a+2b$ (the length of a non-stitch followed by a stitch followed by a non-stitch).  As the former is impossible, we conclude that $a=2b$.  The condition $\gcd(a,b)=1$ implies that $a=2$ and $b=1$. Moreover, this means that the east/west stitch incident to $x$ must extend to the west.  In other words, the three stitches incident to $x$ all meet at $120^\circ$ angles.  Since $x$ was arbitrary, we conclude that all stitches in $H$ meet at $120^\circ$ angles.
\end{proof}

We can leverage the previous two lemmas to show that, up to translation, there is a unique $(a,b)$-triangular hitomezashi pattern.

\begin{theorem}[Classification theorem, $(a,b)$-triangular]\label{thm:classification-triangular}\hspace{.1cm}

\begin{enumerate}
    \item If $1\leq b<a$ are coprime integers such that $(a,b) \neq (2,1)$, then there are no $(a,b)$-hitomezashi patterns.
    \item There are exactly $3$ distinct $(2,1)$-triangular hitomezashi patterns, and they differ only by translations.
\end{enumerate}
\end{theorem}
\begin{proof}
Part (1) of the theorem follows immediately from Lemma~\ref{lem:a=2b}.  We now consider the case where $a=2$ and $b=1$.  It follows from Lemma~\ref{lem:a=2b} that for each fixed stitch $s$ of length $2$ in the triangular grid, there is at most a single $(2,1)$-triangular hitomezashi pattern containing $s$; indeed, $s$ determines the positions of all other stitches because of the $120^\circ$ branching at all stitch endpoints.  Moreover, this procedure produces a valid $(2,1)$-triangular hitomezashi pattern.  See Figure~\ref{fig:triangular-complete}.

Consider a single grid line.  There are three possible ways to place stitches on this grid line.  By the discussion in the previous paragraph, each of the three choices results in a single $(2,1)$-triangular hitomezashi pattern.  These three patterns are clearly distinct (the stitches on our chosen grid line are in different positions) and differ only by translations.
\end{proof}

\begin{figure}[ht]
  \begin{center}\includegraphics[height=6cm]{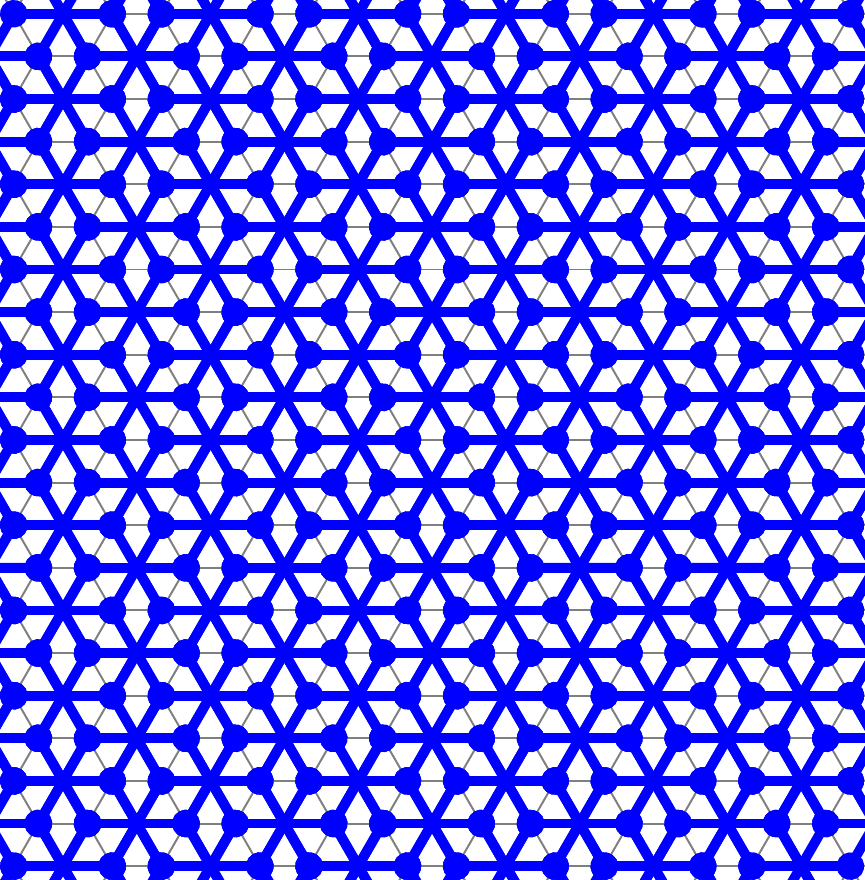}
  \end{center}
  \caption{The unique (up to translation) $(2,1)$-triangular hitomezashi pattern.}\label{fig:triangular-complete}
\end{figure}

\section{Further directions}\label{sec:further}

\subsection{Maximum length for all widths and heights}
Our Theorem~\ref{thm:max-length otherwise} exactly determines the maximum possible length of a hitomezashi loop $L$ with width $w$ and height $h$ for many pairs $(w,h)$ of natural numbers, but the case where both $w$ and $h$ are $3$ modulo $4$ remains open.  We suspect that in this case, the loops maximizing length look like ``combs with one tooth missing.'' 

\subsection{Short stitching on the triangular grid}\label{sec:conclusion-triangular}

A \dfn{triangular hitomezashi pattern} is an arrangement of unit-length stitches in the triangular grid such that every lattice point is the endpoint of exactly one east/west stitch, one northeast/southwest stitch, and one northwest/southeast stitch. These patterns are the same as the $(1,1)$-triangular hitomezashi patterns defined in Section~\ref{sec:long-triangular}. We do not fully understand what the finite connected components in a triangular hitomezashi pattern can look like, but we have formulated a conjecture based on small examples, three of which appear in Figure~\ref{fig:triangular-small}. 

\begin{figure}[ht]
  \begin{center}$\raisebox{0.7cm}{\includegraphics[height=1.466cm]{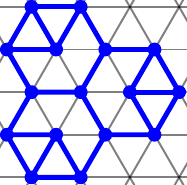}}$\qquad\qquad\includegraphics[height=3.5cm]{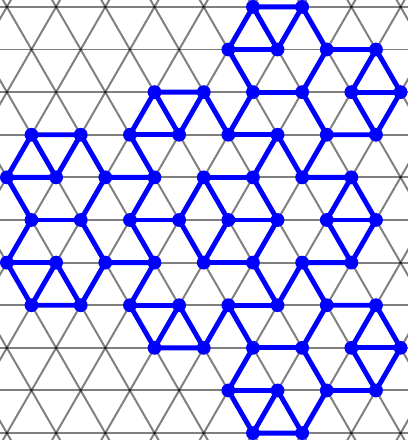}\qquad\qquad\includegraphics[height=3.5cm]{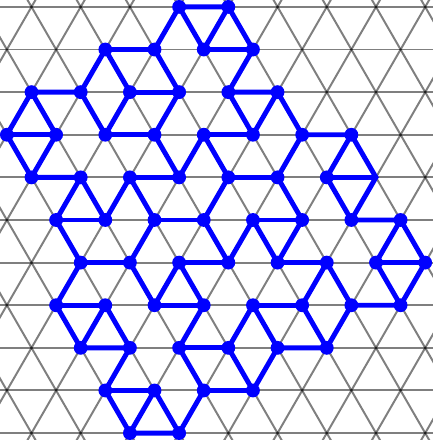}
  \end{center}
  \caption{Three finite connected components that can appear in triangular hitomezashi patterns.}\label{fig:triangular-small}
\end{figure}

\begin{conjecture}\label{conj:divisible_16}
The number of vertices in a finite connected component of a triangular hitomezashi pattern must be divisible by $16$. 
\end{conjecture}

Each vertex in a triangular hitomezashi pattern is of one of the following eight types: 
\[\begin{array}{l} \includegraphics[height=0.788cm]{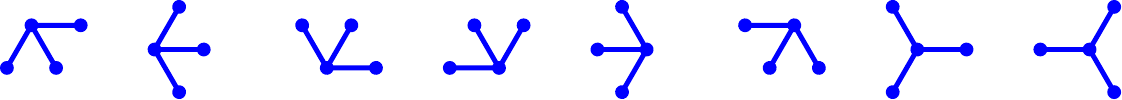}
\end{array}\]
It could be fruitful to consider the distributions of these different types of vertices in finite connected components. 

\subsection{Other lattices}
Consider a nonempty finite set $\Phi\subseteq\mathbb Z^2\setminus\{(0,0)\}$ of nonzero vectors such that $\mathbb R\alpha\cap\Phi=\{\alpha\}$ for every $\alpha\in\Phi$. Define a \dfn{hitomezashi pattern of type $\Phi$} to be a collection $H$ of line segments in $\mathbb R^2$ such that 
\begin{itemize}
    \item every line segment $[u,v]$ in $H$ is such that $u,v\in\mathbb Z^2$ and $v-u\in\pm\Phi$;
    \item for every $u\in\mathbb Z^2$ and every $\alpha\in\Phi$, exactly one of the line segments $[u,u-\alpha]$ and $[u,u+\alpha]$ is in $H$. 
\end{itemize}
Note that a hitomezashi pattern of type $\{(0,1),(1,0)\}$ is the same as an ordinary hitomezashi pattern. We can also realize triangular hitomezashi patterns in this framework. Suppose one of the equilateral triangular cells in the triangular grid has vertices $(0,0)$, $(1,0)$, and $(1/2,\sqrt 3/2)$. Under the linear automorphism of $\mathbb R^2$ defined by $(1,0)\mapsto (1,0)$ and $(1/2,\sqrt 3/2)\mapsto (1,1)$, triangular hitomezashi patterns correspond precisely to hitomezashi patterns of type $\{(1,0),(0,1),(1,1)\}$. 

It would be interesting to see what kinds of structure can emerge in hitomezashi patterns of type $\Phi$ for various choices of $\Phi$. In particular, we can view a hitomezashi pattern $H$ of type $\Phi$ as the set of edges in an infinite graph with vertex set $\mathbb Z^2$. When we speak about the \dfn{connected components} of $H$, we mean the connected components of this graph.

\begin{figure}[ht]
  \begin{center}\includegraphics[height=3.5cm]{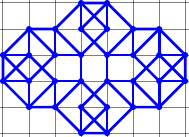}
  \end{center}
  \caption{A finite connected component in a hitomezashi pattern of type $\{(1,0),(0,1),(1,1),(-1,1)\}$. }\label{fig:finite component}
\end{figure}

\begin{question}\label{quest:finite}
Suppose $\Phi\subseteq\mathbb Z^2\setminus\{(0,0)\}$ is a nonempty, finite set such that $\mathbb R\alpha\cap\Phi=\{\alpha\}$ for every $\alpha\in\Phi$. Does there exist a hitomezashi pattern of type $\Phi$ that has a finite connected component? 
\end{question}

When $\Phi$ is $\{(1,0),(0,1)\}$, $\{(1,0),(0,1),(1,1)\}$, or $\{(1,0),(0,1),(1,1),(-1,1)\}$, the answer to Question~\ref{quest:finite} is affirmative; in the last case, Figure~\ref{fig:finite component} shows the smallest (and essentially only) finite connected component that we were able to find. 

\section*{Acknowledgements}
We thank an anonymous referee for pointing out an issue with an earlier version of Theorem~\ref{thm:max-length otherwise}.


\begin{thebibliography}{99}

\bibitem{defant kravitz} C. Defant and N. Kravitz, Loops and regions in hitomezashi patterns. \emph{Preprint} arXiv:2201.03461.

\bibitem{Numberphile}
B. Haran [Numberphile], Hitomezashi stitch patterns -- Numberphile. \emph{YouTube}, (2021).\\ \url{https://www.youtube.com/watch?v=JbfhzlMk2eY&t=1s&ab_channel=Numberphile}. Accessed December 6, 2021. 

\bibitem{HayesSeaton2}
C. Hayes and K. A. Seaton, Mathematical specification of hitomezashi designs. \emph{Preprint} arXiv:2208.12580. 


\bibitem{HayesSeaton}
C. Hayes and K. A. Seaton, A two-dimensional introduction to sashiko. \emph{Bridges Conference Proceedings}, (2020), 517--524.

\bibitem{Holden} J. Holden, The graph theory of blackwork embroidery.  \emph{Making Mathematics with Needlework} (2005), 137--154.

\bibitem{Pete}
G. Pete, Corner percolation on $\mathbb{Z}^2$ and the square root of $17$.  \emph{Ann. Prob.}, {\b36} (2008), 1711--1747.
\end{thebibliography}
\end{document}